\documentclass[leqno,11pt]{amsart}
\usepackage{amsmath,amstext,amssymb,amsopn,amsthm,mathrsfs}

\usepackage{graphicx}

\allowdisplaybreaks

\newtheorem{theor}{Theorem}[section]

\newtheorem{lemma}[theor]{Lemma}

\newtheorem*{rem*}{Remark}

\newcommand{\R}{\mathbb{R}^d}

\begin{document}
%%%%%%%%%%%%%%%%%%%%%%%%%%%%%%%%%%%%%%%%%%%%
\footnotetext{
\\
%47G40 - Operator theory: potential operators
%31C15 - Potential theory: potentials and capacities
Research of both authors supported by MNiSW Grant N N201 417839.
}

\title[Sharp estimates of the potential kernel]
	{Sharp estimates of the potential kernel for the harmonic oscillator with applications}

%%%%%%%%%%%%%%%%%%%%%%%%%%%%%%%%%%%%%%%%%%%%%%

\author[A. Nowak]{Adam Nowak}
\address{Adam Nowak, \newline
			Instytut Matematyczny,
      Polska Akademia Nauk, \newline
      \'Sniadeckich 8,
      00--956 Warszawa, Poland
      }
\email{Adam.Nowak@impan.pl}

\author[K. Stempak]{Krzysztof Stempak}
\address{Krzysztof Stempak,     \newline
      Instytut Matematyki i Informatyki,
      Politechnika Wroc\l{}awska,       \newline
      Wyb{.} Wyspia\'nskiego 27,
      50--370 Wroc\l{}aw, Poland,       \newline
      \indent and \newline
      Katedra Matematyki i Zastosowa\'n Informatyki,
      Politechnika Opolska,       \newline
      Miko\l{}ajczyka 5,
      45--271 Opole, Poland}
\email{Krzysztof.Stempak@pwr.wroc.pl}

%%%%%%%%%%%%%%%%%%%%%%%%%%%%%%%%%%%%%%%%%%%%%%
\begin{abstract}
We prove qualitatively sharp estimates of the potential kernel for the harmonic oscillator.
These bounds are then used to show that the $L^p-L^q$ estimates of the associated potential
operator obtained recently by Bongioanni and Torrea \cite{BT} are in fact sharp.
\end{abstract}
%%%%%%%%%%%%%%%%%%%%%%%%%%%%%%%%%%%%%%%%%%%%%%

\maketitle

%%%%%%%%%%%%%%%%%%%%%%%%%%%%%%%%%%%%%%%%%%%%%%
\section{Introduction} \label{sec:intro}
%%%%%%%%%%%%%%%%%%%%%%%%%%%%%%%%%%%%%%%%%%%%%%
The study of the potential theory for the $d$-dimensional harmonic oscillator 
$$
\mathcal{H}=-\Delta+\|x\|^2,
$$
has recently been initiated by Bongioanni and Torrea \cite{BT}. 
The multi-dimensional Hermite functions $h_k$ are eigenfunctions of $\mathcal{H}$ and we have
$\mathcal{H} h_k = (2|k|+d) h_k$. The operator $\mathcal{H}$ has a natural self-adjoint
extension, here still denoted by $\mathcal{H}$, whose spectral decomposition is given by the $h_k$.

The integral kernel $G_t(x,y)$ of the Hermite semigroup
$\{\exp({-t\mathcal{H}}): t>0\}$ is known explicitly to be
(see \cite{ST2} for this symmetric variant of the formula)
\begin{align*}
G_t(x,y)&=\sum_{n=0}^\infty e^{-(2n+d)t}\sum_{|k|=n}h_k(x)h_k(y)\\
&=\big(2\pi\sinh(2t)\big)^{-d/2}\exp\bigg(-\frac14\Big[\tanh( t)\|x+y\|^2+\coth(t)\|x-y\|^2 \Big]\bigg).
\end{align*}

Given $\sigma>0$, consider the negative power $\mathcal{H}^{-\sigma}$, which is a contraction on $L^2(\R)$.
It is easily seen that $\mathcal{H}^{-\sigma}$ coincides in $L^2(\R)$ with the \textit{potential operator}
\begin{equation}\label{int}
\mathcal{I}^\sigma f(x)=\int_{\R}\mathcal{K}^{\sigma}(x,y)f(y)\,dy,
\end{equation}
where the \textit{potential kernel} is given by
\begin{equation}\label{ker}
\mathcal{K}^{\sigma}(x,y)= \frac1{\Gamma(\sigma)}\int_0^\infty G_t(x,y) t^{\sigma-1}\,dt.
\end{equation}
Note that all the spaces $L^p(\mathbb R^d)$, $1\le p\le\infty$, are contained in
the natural domain of $\mathcal{I}^\sigma$ consisting of those functions $f$ for which the 
integral in \eqref{int} converges $x$-a.e., see \cite[Section 2]{NS3}. 

The main result of the paper, Theorem \ref{main1} below, provides qualitatively sharp estimates
of the potential kernel \eqref{ker}. As an application of this result, we prove
sharpness of the $L^p-L^q$ estimates for the potential
operator \eqref{int} obtained recently by Bongioanni and Torrea \cite[Theorem 8]{BT},
see Theorem~\ref{thm:LpLq}.

Recall that an operator $T$ defined on $L^p(\mathbb R^d)$ for some $1\le p\le\infty$, 
with values in the space of measurable functions on $\mathbb{R}^d$, is said to be 
of weak type $(p,q)$, $1\le q<\infty$, provided that
\begin{equation}\label{weak}
|\{x\in \mathbb R^d\colon |Tf(x)|>\lambda\}|\le C \Big(\|f\|_p\slash \lambda\Big)^q,
\end{equation}
with $C>0$ independent of $f\in L^p(\mathbb R^d)$ and $\lambda>0$. 
The restricted weak type $(p,q)$ of $T$ means that \eqref{weak} holds for $f=\chi_E$,
where $E$ is any measurable subset of $\mathbb R^d$ of finite measure.
By definition, weak type $(p,\infty)$ coincides with strong type $(p,\infty)$, 
{i.e{.}} the estimate $\|Tf\|_\infty\le C\|f\|_p$, $f\in L^p(\mathbb R^d)$. 
In terms of Lorentz spaces, the weak type $(p,q)$ is equivalent to the boundedness from
$L^p(\mathbb{R}^d)$ to $L^{q,\infty}(\mathbb{R}^d)$, and the restricted weak type $(p,q)$ is
characterized by the boundedness from $L^{p,1}(\mathbb{R}^d)$ to $L^{q,\infty}(\mathbb{R}^d)$,
see \cite[Chapter 4, Section 4]{BS}. Strong type $(p,q)$ means of course the $L^p$-$L^q$ boundedness.

The notation $X\lesssim Y$ will be used to indicate that $X\leq CY$ with a positive constant $C$
independent of significant quantities; we shall write $X \simeq Y$ when simultaneously 
$X \lesssim Y$ and $Y \lesssim X$.  We will also use the notation $X\simeq\simeq  Y\exp(-cZ)$ 
to indicate that there exist positive constants $C, c_1$ and $c_2$, independent of significant
quantities, such that
$$
C^{-1}\,Y\exp(-c_1Z)\le X\le C\,Y\exp(-c_2Z).
$$
Further, in a number of places, we will use natural and self-explanatory generalizations
of the $\simeq \simeq$ relation, for instance in connection with certain integrals 
involving exponential factors. In such cases the exact meaning will be clear from the context. 
By convention, $\simeq \simeq$ is understood as $\simeq$ whenever there are no exponential
factors involved.

We write $\log^+$ for the positive part of the logarithm, and $\vee,\wedge$ for the operations 
of taking maximum and minimum, respectively.

%%%%%%%%%%%%%%%%%%%%%%%%%%%%%%%%%%%%%%%%%%%%%%
\section{Estimates of the potential kernel} \label{sec:asym}
%%%%%%%%%%%%%%%%%%%%%%%%%%%%%%%%%%%%%%%%%%%%%%

We begin with two technical results describing the behavior of the integrals
\begin{align*}
 I_A(T) & = \int_T^{\infty}t^A\exp(-t)\,dt, \qquad T>0, \\
 J_A(T,S) & = \int_T^St^A\exp(-t)\,dt, \qquad 0<T<S<\infty.
\end{align*}
Notice that $I_A(T)$ dominates $J_A(T,S)$.
The lemma below is a refinement of \cite[Lemma 2.1]{NS3}, see also \cite[Lemma 1.1]{ST1}.
%%%%%%%%%%%%%%%%%%%%%%%%%%%%%%%
\begin{lemma}
\label{lem:le1}
Let $A\in\mathbb R$ and $\gamma > 0$ be fixed. Then
\begin{equation} \label{estea}
 I_A(\gamma T)\simeq T^A\exp(-\gamma T), \qquad T \ge 1,
\end{equation}
and for $0<T<1$
\begin{equation*}
 I_A(\gamma T)\simeq  
\begin{cases}
T^{A+1}, & \quad A<-1 \\
\log (2\slash T), & \quad A=-1 \\
1, & \quad A>-1
\end{cases} \;\;.
\end{equation*}
\end{lemma}
%%%%%%%%%%%%%%%%%%%%%%%%%%%%%%%
\begin{proof}
We assume that $\gamma =1$. From the proof it will be clear that the estimates are true for any $\gamma >0$.
The case $0<T<1$ was treated in the proof of \cite[Lemma 2.1]{NS3}, so we consider $T \ge 1$
and focus on showing \eqref{estea}. The lower bound in \eqref{estea} is straightforward, we have
\begin{equation*}
I_A(T) > \int_T^{2T}t^Ae^{-t}\,dt\gtrsim T^A \int_T^{2T} e^{-t} \, dt =
T^A\big(e^{-T}-e^{-2T}\big)\gtrsim T^A e^{-T}, \qquad T \ge 1.
\end{equation*}
It remains to prove the upper bound,
\begin{equation}\label{aa}
\int_T^\infty t^Ae^{-t}\,dt\lesssim  T^A e^{-T}, \qquad T \ge 1,
\end{equation}
and here we assume that $A>0$, since for $A\le0$ we have $t^A \le T^{A}$, $t > T \ge 1$, and
the conclusion is trivial. Choosing $T_A$ such that for $T\ge T_A$ one has
\begin{equation*}
\int_{2T}^\infty t^Ae^{-t}\,dt\le \frac12\int_{T}^\infty t^Ae^{-t}\,dt,
\end{equation*}
we can write
\begin{equation*}
\int_T^\infty t^Ae^{-t}\,dt\le \int_T^{2T} t^Ae^{-t}\,dt+\int_{2T}^\infty t^Ae^{-t}\,dt\le
 C\,T^Ae^{-T}+\frac12\int_{T}^\infty t^Ae^{-t}\,dt, \qquad T \ge T_A.
\end{equation*}
This implies \eqref{aa} for $T\ge T_A$ and consequently for all $T\ge1$.
\end{proof}

%%%%%%%%%%%%%%%%%%%%%%%%%%%%%%%
\begin{lemma}
\label{lem:le2}
Let $A\in\mathbb R$ and $\gamma > 0$ be fixed. Then for $0<T<S\le 2T$ we have
\begin{equation} \label{estea2}
T^A(S-T)\exp(-2\gamma T) \lesssim J_A(\gamma T,\gamma S)\lesssim T^A(S-T)\exp(-\gamma T),
\end{equation}
while for $S>2T>0$ we have $J_A(\gamma T,\gamma S)\simeq I_A(\gamma T)$ when $S\ge2$, and
\begin{equation*}
 J_A(\gamma T,\gamma S)\simeq  
\begin{cases}
 T^{A+1}, & \quad A<-1\\
\log (S\slash T), & \quad A=-1 \\
S^{A+1}, & \quad A>-1
\end{cases}\;\; 
\end{equation*}
when $0<S<2$.
\end{lemma}
%%%%%%%%%%%%%%%%%%%%%%%%%%%%%%%
\begin{proof}
As in the proof of Lemma \ref{lem:le1}, it is enough to deal with the case $\gamma =1$.
The bounds for $T<S\le 2T$ follow since then $\int_T^St^Ae^{-t}\,dt\simeq T^A\int_T^Se^{-t}\,dt$ and
$$
(S-T)e^{-2T}\le \int_T^Se^{-t}\,dt\le (S-T)e^{-T}. 
$$

Assume now that $S > 2T$. Clearly, $J_A(T,S) < I_A(T)$. On the other hand, if $T \ge 1$ then 
$$
J_A(T,S) > \int_T^{2T} t^A e^{-t}\, dt \gtrsim T^A \int_T^{2T} e^{-t}\, dt \gtrsim T^A e^{-T}
\gtrsim I_A(T),
$$
the last estimate being a consequence of \eqref{estea}. When $0< T < 1$, we distinguish two subcases.
If $S \ge 2$, then again $J_A(T,S) \gtrsim \int_T^2 t^A\, dt \gtrsim I_A(T)$.
If $2T<S<2$, then $J_A(T,S) \simeq \int_T^S t^A \, dt$, and evaluating the last integral we arrive
at the claimed bounds for $J_A(T,S)$.
\end{proof}

We note that \eqref{estea} and \eqref{estea2} may be written slightly less precisely as 
\begin{align*}
I_A(\gamma T) & \simeq\simeq \exp(-cT), \qquad T \ge 1, \\
J_A(\gamma T,\gamma S) & \simeq\simeq T^A (S-T)\exp(-cT), \qquad 0<T<S\le 2T,
\end{align*}
respectively. This fact will be used in the sequel without further mention.

We now apply Lemmas \ref{lem:le1} and \ref{lem:le2} to prove qualitatively sharp estimates of the integral
\begin{equation*}
 E_A(T,S)=\int_0^1t^A\exp\big(-Tt^{-1}-St\big)\,dt, \qquad 0<T,S<\infty. 
\end{equation*}
The following result provides, in particular, a refinement and generalization of \cite[Lemma 2.4]{NS2}.
%%%%%%%%%%%%%%%%%%%%%%%%%%
\begin{lemma} \label{pro:comp}
Let $A \in \mathbb{R}$ be fixed. Then
\begin{equation*}
E_A(T,S)\simeq\simeq\exp\Big(-c\sqrt{T(T\vee S)}\Big)\times \begin{cases}
T^{A+1}, & \quad A<-1 \\
1+\log^+\frac1{T(T\vee S)}, & \quad A=-1 \\
(S\vee 1)^{-A-1}, & \quad A>-1
\end{cases}\;\; ,
\end{equation*}
uniformly in $T,S>0$.
\end{lemma}
%%%%%%%%%%%%%%%%%%%%%%%%%%%%%%%
\begin{proof}
We first estimate $E_A(T,S)$ in terms of the integrals $I_A$ and $J_A$.
For $0<S\le 2T$ we have
\begin{equation*}
E_A(T,S)\simeq\simeq \int_0^1 t^A\exp(-cTt^{-1})\,dt
	\simeq T^{A+1}\int_{cT}^\infty u^{-A-2}e^{-u}\,du=T^{A+1}I_{-A-2}(cT),
\end{equation*}
where the second relation follows by the change of variable $t=cT\slash u$.
When $S>2T$ we change the variable $t=u\sqrt{T\slash S}$ and get
$$
E_A(T,S)=\Big(\frac TS\Big)^{(A+1)\slash2}
\int_0^{\sqrt{S\slash T}}u^A\exp\big(-\sqrt{TS}(u+u^{-1})\big)\,du\equiv \mathcal J_1+\mathcal J_2,
$$
where $\mathcal J_1$ and $\mathcal J_2$ come from splitting the integration over 
the intervals $(0,1)$ and $(1,\sqrt{S/T})$, respectively. Then
\begin{align*}
\mathcal J_1\simeq\simeq \Big(\frac TS\Big)^{(A+1)\slash2}
\int_0^1 u^A\exp\big(-c\sqrt{TS}u^{-1}\big)\,du&\simeq T^{A+1}
\int_{c\sqrt{TS}}^\infty z^{-A-2}e^{-z}\,dz \\
& =T^{A+1}I_{-A-2}\big(c\sqrt{TS}\big)
\end{align*}
and
\begin{align*}
\mathcal J_2\simeq\simeq \Big(\frac TS\Big)^{(A+1)\slash2}
\int_1^{\sqrt{S\slash T}} u^A\exp\big(-c\sqrt{TS}u\big)\,du&
\simeq S^{-A-1}\int_{c\sqrt{TS}}^{cS} z^{A}e^{-z}\,dz\\
&=S^{-A-1}J_{A}\big(c\sqrt{TS},cS\big).
\end{align*}
Summing up, we have
\begin{equation*}
E_A(T,S)\simeq\simeq T^{A+1}I_{-A-2}\big(c\sqrt{T(T\vee S)}\big)
	+ \chi_{\{S>2T\}}S^{-A-1}J_{A}\big(c\sqrt{TS},cS\big),
\end{equation*}
uniformly in $S,T>0$. In the next step we describe the behavior of the two terms here by means 
of Lemmas \ref{lem:le1} and \ref{lem:le2}.

From Lemma \ref{lem:le1} it follows that
$$
T^{A+1}I_{-A-2}\big(c\sqrt{T(T\vee S)}\big) \simeq\simeq 
	T^{A+1}\exp\big(-c\sqrt{T(T\vee S)}\big), \qquad T(T\vee S)\ge1,
$$
(here, and also in analogous places below, $c$ on the left-hand side should be understood 
as a \textit{given} constant) and
\begin{equation*}
T^{A+1}I_{-A-2}\big(c\sqrt{T(T\vee S)}\big)\simeq
\begin{cases}
T^{A+1}, & \quad A<-1 \\
\log (\frac4{T(T\vee S)}), & \quad A=-1 \\ 
\big(\frac T{T\vee S}\big)^{(A+1)\slash2}, & \quad A>-1
\end{cases} \;\; , \qquad T(T\vee S)\le1.
\end{equation*}
The term $S^{-A-1}J_{A}(c\sqrt{TS},cS)$ comes into play when $S>2T$, and in this case 
we use Lemma \ref{lem:le2} to write the bounds
$$
S^{-A-1}J_{A}\big(c\sqrt{TS},cS\big) \simeq \chi_{\{S\ge 2\}}\Phi_1+\chi_{\{S<2\}}\Phi_2,
$$
where
$$
\Phi_1 = S^{-A-1}I_{A}\big(c\sqrt{TS}\big), \qquad \qquad
\Phi_2=
\begin{cases}
(T\slash S)^{(A+1)\slash2}, & \quad A<-1 \\
\log (\frac ST), & \quad A=-1 \\ 
1, & \quad A>-1
\end{cases}\;\; .
$$
By Lemma \ref{lem:le1},
\begin{align*}
\Phi_1 & \simeq\simeq S^{-A-1}\exp\big(-c\sqrt{TS}\big), \qquad TS\ge1, \\
\Phi_1 & \simeq 
\begin{cases}
(T\slash S)^{(A+1)\slash2}, & \quad A<-1\\
\log (\frac 4{TS}), & \quad A=-1\\ 
S^{-A-1}, & \quad A>-1
\end{cases}\;\; , \qquad TS\le1.
\end{align*}
To proceed, it is convenient to consider each of the cases  $A<-1$, $A=-1$, and $A>-1$ separately. 

If $A<-1$, then
\begin{align*}
E_{A}(T,S) & \simeq\simeq \chi_{\{2>S> 2T\}} \Big(\frac TS\Big)^{(A+1)\slash2} +
\begin{cases}
T^{A+1}\exp\big(-c\sqrt{T(T\vee S)}\big), &\quad T(T\vee S)\ge1\\
T^{A+1}, &\quad T(T\vee S)<1\\ 
\end{cases} \\
& \qquad
+\chi_{\{S>2T\}} \chi_{\{S\ge 2\}}
\begin{cases}
T^{A+1}\exp\big(-c\sqrt{TS}\big), &\quad TS\ge1\\
\big(\frac TS\big)^{(A+1)\slash2}, &\quad TS<1
\end{cases}\;\; .
\end{align*}
Here the first and third terms are insignificant in comparison to the second one.
In case of the third summand, this is because $A<-1$ and
$\big(\frac TS\big)^{(A+1)\slash2}<T^{A+1}$ for $TS<1$. 
A similar argument is used for the first one.
The required estimates of $E_{A}(T,S)$ follow. 
 
If $A=-1$, then
\begin{align*}
E_{-1}(T,S) & \simeq\simeq \chi_{\{2>S> 2T\}} \log \frac ST +
\begin{cases}
\exp\big(-c\sqrt{T(T\vee S)}\big), &\quad T(T\vee S)\ge1\\
\log \big(\frac4{T(T\vee S)}\big), &\quad T(T\vee S)<1
\end{cases} \\
& \qquad +\chi_{\{S>2T\}} \chi_{\{S\ge 2\}}
\begin{cases}
\exp\big(-c\sqrt{TS}\big), &\quad TS\ge1\\
\log \big(\frac4{TS}\big), &\quad TS<1\\ 
\end{cases}\;\; .
\end{align*}
Similarly as in the case of $A<-1$, here also the first and third terms are insignificant 
in comparison to the second one. This is clear for the third summand, and
for the first one this is because $\log \frac ST < \log (\frac4{TS})$ when $S<2$. 
Thus the desired bounds of $E_{-1}(T,S)$ also follow. 
  
Finally, we consider the case $A>-1$, which is less direct than the previous two.
We have
\begin{align*}
E_{A}(T,S) & \simeq\simeq \chi_{\{2>S> 2T\}} + 
\begin{cases}
T^{A+1}\exp\big(-c\sqrt{T(T\vee S)}\big), &\quad T(T\vee S)\ge1\\
 \big(\frac T{T\vee S}\big)^{(A+1)\slash2} , &\quad T(T\vee S)<1\\ 
\end{cases} \\
& \qquad +\chi_{\{S>2T\}} \chi_{\{S\ge 2\}} 
\begin{cases}
T^{A+1}\exp\big(-c\sqrt{TS}\big), &\quad TS\ge1\\
S^{-A-1}, &\quad TS<1\\ 
\end{cases}\;\; .
\end{align*}
Observe that here the relation $\simeq \simeq$ remains valid if the sum of the first and the third
terms is replaced by the comparable (in the sense of $\simeq$) expression 
\begin{equation*}
\chi_{\{S>2T\}}
\begin{cases}
T^{A+1}\exp\big(-c\sqrt{TS}\big), &\quad TS\ge1\\
 (S\vee 1)^{-A-1} , &\quad TS<1\\ 
\end{cases} \;\; .
\end{equation*}
Taking into account that $T^{A+1}\exp(-c\sqrt{TS})\simeq\simeq S^{-A-1}\exp(-c\sqrt{TS})$ for
$TS\ge1$,
we conclude that 
\begin{align*}
E_{A}(T,S) & \simeq\simeq 
\begin{cases}
(T\vee S)^{-A-1}\exp\big(-c\sqrt{T(T\vee S)}\big), &\quad T(T\vee S)\ge1\\
 \big(\frac T{T\vee S}\big)^{(A+1)\slash2} , &\quad T(T\vee S)<1\\ 
\end{cases} \\
& \qquad + \chi_{\{S>2T\}}
\begin{cases}
S^{-A-1}\exp\big(-c\sqrt{TS}\big), &\quad TS\ge1\\
(S\vee 1)^{-A-1}, &\quad TS<1 
\end{cases}\;\; .
\end{align*}
Now, if $T\ge S$ and $T(T\vee S)=T^2<1$, then $(\frac T{T\vee S}\big)^{1\slash2}=1\simeq 1\slash (S\vee 1)$,
while for $T < S$ and $T(T\vee S)=TS<1$, we have 
$(\frac T{T\vee S}\big)^{1\slash2}=(\frac TS\big)^{1\slash2} < 1\slash (S\vee 1)$. Therefore,
\begin{equation*}
E_{A}(T,S)\simeq\simeq 
\begin{cases}
(T\vee S)^{-A-1}\exp\big(-c\sqrt{T(T\vee S)}\big), &\quad T(T\vee S)\ge1\\
 (S\vee 1)^{-A-1},  &\quad T(T\vee S)<1 
\end{cases}\;\; .
\end{equation*}
We claim that this implies
\begin{equation*}
E_{A}(T,S)\simeq\simeq (S\vee 1)^{-A-1}\exp\big(-c\sqrt{T(T\vee S)}\big), 
\end{equation*}
which are precisely the required estimates.

To justify the claim, it is enough to recall that $A>-1$ and observe that if $T\ge S$ and 
$T(T\vee S) = T^2 \ge 1$, then 
\begin{align*}
(T\vee S)^{-A-1}\exp\big(-c\sqrt{T(T\vee S)}\big)=T^{-A-1}\exp(-cT)&\simeq (T\vee 1)^{-A-1}\exp(-cT)\\
&\simeq\simeq (S\vee 1)^{-A-1}\exp(-cT),
\end{align*}
while if $T < S$ and $T(T\vee S)=TS\ge 1$ (this forces $S>1$), then 
$$
(T\vee S)^{-A-1}\exp\big(-c\sqrt{T(T\vee S)}\big)=
S^{-A-1}\exp\big(-c\sqrt{TS}\big)\simeq (S\vee 1)^{-A-1}\exp\big(-c\sqrt{TS}\big).
$$

The proof is finished.
\end{proof}

We are now in a position to prove qualitatively sharp estimates of the potential kernel.
%%%%%%%%%%%%%%%%%%%%%%%%%%%%%%%
\begin{theor}\label{main1}
For $\sigma>0$ we have
\begin{equation*}
\mathcal K^\sigma(x,y)\simeq\simeq\exp\big(-c\|x-y\|(\|x\|+\|y\|)\big)\times 
\begin{cases}
\|x-y\|^{2\sigma-d}, &\quad \sigma<d\slash2 \\
1+\log^+\frac1{\|x-y\|(\|x\|+\|y\|)}, &\quad \sigma=d\slash2 \\
(1+\|x+y\|)^{d-2\sigma}, &\quad\sigma>d\slash2
\end{cases}\;\; ,
\end{equation*}
uniformly in $x,y\in\mathbb R^d$.
\end{theor}
%%%%%%%%%%%%%%%%%%%%%%%%%%
\begin{proof}
We decompose
$$
\Gamma(\sigma)\mathcal K^\sigma(x,y)=\int_0^1G_t(x,y)\,t^{\sigma-1}\,dt+
\int_1^\infty G_t(x,y)\,t^{\sigma-1}\,dt\equiv \mathcal J^\sigma_0(x,y)+\mathcal J^\sigma_\infty(x,y).
$$
For $0<t<1$ we have $\tanh t \simeq t$, $\coth t \simeq t^{-1}$, $\sinh 2t \simeq t$, and therefore 
$$
\mathcal J^\sigma_0(x,y)\simeq\simeq E_{\sigma-d\slash2-1}\big(c\|x-y\|^2,c\|x+y\|^2\big). 
$$
This combined with Lemma \ref{pro:comp} shows that the estimates from the statement hold with
$\mathcal{K}^{\sigma}(x,y)$ replaced by $\mathcal{J}^{\sigma}_0(x,y)$. 
Further, taking into account that $\tanh t \simeq 1 \simeq \coth t$ for $t > 1$, we see that
$$
\mathcal J^\sigma_\infty(x,y) \simeq\simeq \exp\big(-c(\|x\|^2+\|y\|^2)\big).
$$
Thus $\mathcal J^\sigma_0(x,y)$ dominates $\mathcal J^\sigma_\infty(x,y)$ 
in the above decomposition, in the sense that
$$
\mathcal J^\sigma_\infty(x,y) \lesssim E_{\sigma-d\slash2-1}\big(c\|x-y\|^2,c\|x+y\|^2\big)
$$
for a sufficiently small constant $c>0$. The conclusion follows.
\end{proof}

%%%%%%%%%%%%%%%%%%%%%%%%%%%%%%%%%%%%%%%%%%%%%%%%%%
\section{Sharpness of the $L^p$-$L^q$ boundedness of the potential operator} \label{sec:sharp}
%%%%%%%%%%%%%%%%%%%%%%%%%%%%%%%%%%%%%%%%%%%%%%%%%%
Given $0<\sigma<d/2$, define the region
\begin{align*}
R & =\bigg\{\Big(\frac1p,\frac1q\Big)\colon 0\le\frac1p\le1\,\,\,{\rm and}\,\,\, 0
\vee \Big(\frac1p-\frac{2\sigma}d\Big)\le\frac1q\le 1 \wedge\Big(\frac1p+\frac{2\sigma}d\Big)\bigg\} \\
& \qquad \Big\backslash 
\bigg(\bigg\{\Big(\frac1p,\frac1q\Big)\colon 0\le\frac1p\le1-\frac{2\sigma}d\,\,\,{\rm and}\,\,\, 
 \frac1q = \frac1p+\frac{2\sigma}d\bigg\} \cup 
 	\bigg\{\Big(\frac{2\sigma}d,0\Big),\Big(1,1-\frac{2\sigma}d\Big)\bigg\}\bigg)
\end{align*}
contained in the unit $(\frac1p,\frac1q)$-square $[0,1]^2$, see Figure \ref{fig1}. 

% ********************* FIGURE
\begin{figure}[ht]
\includegraphics[width=0.6\textwidth]{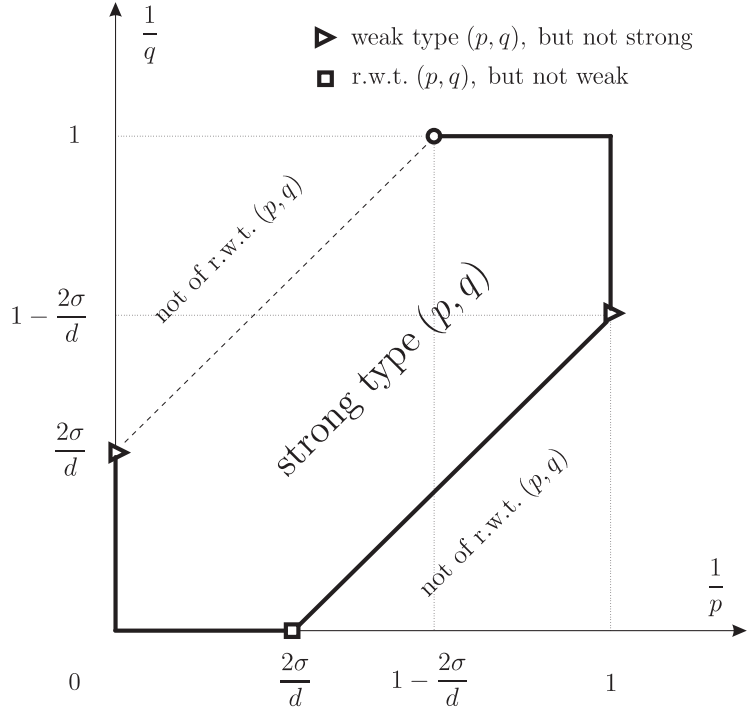}
\caption{Mapping properties of $\mathcal{I}^{\sigma}$ for $0 < \sigma < d/2$.}\label{fig1}
\end{figure}
% ********************* FIGURE

The following result enhances \cite[Theorem 8]{BT}, see also \cite[Theorem 2.3]{NS3}. 
%%%%%%%%%%%%%%%%%%%%%%%%%%
\begin{theor} \label{thm:LpLq}
Let $d\ge1$, $0<\sigma<d\slash2$ and $1\le p,q\le\infty$.
Then $\mathcal{I}^\sigma\colon L^p(\mathbb R^d)\to L^q(\mathbb R^d)$ boundedly 
if and only if $(\frac 1p,\frac1q)$ lies in the region $R$.

On the other hand, 
$\mathcal{I}^{\sigma}$ is not even of restricted weak type $(p,q)$ when $(\frac{1}p,\frac{1}q)$ is
not in the closure of $R$. Moreover, $\mathcal{I}^{\sigma}$ is of weak type $(p,q)$ for
$(\frac{1}p,\frac{1}q) = (0,\frac{2\sigma}d)$ and $(\frac{1}p,\frac{1}q) = (1,1-\frac{2\sigma}d)$.
For $(\frac{1}p,\frac{1}q) = (\frac{2\sigma}d,0)$ the restricted weak type is true, whereas weak type fails.
\end{theor}
%%%%%%%%%%%%%%%%%%%%%%%%%%

Before giving the proof we take the opportunity to present a short argument showing 
\cite[(21) and (41)]{BT}, the result we will apply in a moment.
%%%%%%%%%%%%%%%%%%%%%%%%%%%%%%%
\begin{lemma}
\label{lem:le3}
Given $\sigma>0$, 
\begin{equation*} 
\|\mathcal K^\sigma(x,\cdot)\|_1\simeq (1\vee \|x\|)^{-2\sigma}, \qquad x \in \mathbb{R}^d.
\end{equation*}
\end{lemma}
%%%%%%%%%%%%%%%%%%%%%%%%%%%%%%%
\begin{proof}
Using the identity (see \cite[Proposition 3.3]{ST2})
$$
\exp({-t\mathcal H}) \boldsymbol{1}(x)=\int_{\mathbb R^d}G_t(x,y)\,dy= 
(\cosh 2t)^{-d/2}\exp\Big(-\frac12\tanh(2t)\|x\|^2\Big), \qquad x\in\mathbb R^d,
$$
we may write
\begin{align*}
\int_{\mathbb R^d}\mathcal K^\sigma(x,y)\,dy
& = \frac1{\Gamma(\sigma)}\int_0^\infty\int_{\mathbb R^d}G_t(x,y)\,dy\,t^{\sigma-1}\,dt\\
& = \frac{1}{\Gamma(\sigma)} \int_0^{\infty} (\cosh 2t)^{-d/2}\exp\Big(-\frac12\tanh(2t)\|x\|^2\Big)
	t^{\sigma -1}\, dt.
\end{align*}
Here we split the integration to the intervals $(0,1)$ and $(1,\infty)$ and denote the resulting integrals
by $\mathcal{J}_0$ and $\mathcal{J}_{\infty}$, respectively. Then, uniformly in $x\in \mathbb{R}^d$,
$$
\mathcal{J}_0 \simeq \simeq \int_0^{1} \exp\big(-ct\|x\|^2\big) t^{\sigma-1}\, dt
	= \|x\|^{-2\sigma} \int_0^{\|x\|^2} e^{-ct} s^{\sigma-1}\, dt 
	\simeq \|x\|^{-2\sigma} \big( \|x\|^{2\sigma} \wedge 1\big)
$$
and
$$
\mathcal{J}_{\infty} \simeq \simeq \int_1^{\infty} e^{-td} \exp\big(-c\|x\|^2\big) t^{\sigma-1}\, dt
	= C_{d,\sigma} \exp\big(-c\|x\|^2\big).
$$
The conclusion follows.
\end{proof}

\begin{proof}[Proof of Theorem \ref{thm:LpLq}]
We first focus on strong type inequalities. Then, in view of \cite[Theorem 8]{BT}, what remains 
to prove are the following two items.
\begin{itemize}
\item[(a)] $\mathcal{I}^\sigma$ is not $L^p-L^q$ bounded for
$\frac{2\sigma}d<\frac1p<1$ and $0<\frac1q<\frac1p-\frac{2\sigma}d$.
\item[(b)] $\mathcal{I}^\sigma$ is not $L^p-L^q$ bounded for 
	$0<\frac1p<1-\frac{2\sigma}d$ and $\frac1p+\frac{2\sigma}d\le\frac1q<1$.
\end{itemize}
 
To justify (a), we fix $p$ and $q$ satisfying the assumed conditions and define
$$
f(y) = \chi_{\{\|y\|<1\}} \|y\|^{-2\sigma -d/q}.
$$
This function is in $L^p(\mathbb{R}^d)$ since $-(2\sigma + d/q)p+d>0$. However,
$\mathcal{I}^{\sigma}f \notin L^q(\mathbb{R}^d)$. Indeed, considering $x$ such that $\|x\|<1$ and using the
lower bound from Theorem \ref{main1} we get
$$
\mathcal{I}^{\sigma}f(x) \gtrsim \int_{\|y\|<\|x\|/2} \|x-y\|^{2\sigma-d} \|y\|^{-2\sigma-d/q}\, dy
	\gtrsim \|x\|^{2\sigma-d} \int_{\|y\|<\|x\|/2} \|y\|^{-2\sigma-d/q}\, dy = C \|x\|^{-d/q},
$$
and the function $x \mapsto \chi_{\{\|x\|< 1\}} \|x\|^{-d/q}$ does not belong to $L^q(\mathbb{R}^d)$.

Proving (b) we may assume that $(\frac{1}{p},\frac{1}{q})$ lies on the critical segment 
$\frac1q=\frac1p+\frac{2\sigma}d$, $0<\frac1p<1-\frac{2\sigma}d$. The case
when $\frac1q>\frac1p+\frac{2\sigma}d$ is contained below, in the negative result concerning the
restricted weak type estimate.
Define
$$
f(y) = \chi_{\{\|y\|>e\}} \|y\|^{-d/p} \big( \log\|y\|\big)^{-1/p-2\sigma/d}.
$$
We have 
$$
\int_{\mathbb{R}^d} |f(y)|^p\, dy = C_d \int_e^{\infty} r^{-1} (\log r)^{-1-2\sigma p/d}\, dr < \infty,
$$
so $f \in L^p(\mathbb{R}^d)$. We claim that $\mathcal{I}^{\sigma}f \notin L^q(\mathbb{R}^d)$.
Assuming that $\|x\| > 2e$ and using the lower bound from Theorem \ref{main1} we write
\begin{align*}
\mathcal{I}^{\sigma}f(x) & \gtrsim \int_{\|x\|/2 < \|y\| < \|x\|} \|x-y\|^{2\sigma-d}
	\exp\big(-c\|x-y\|(\|x\|+\|y\|)\big) \|y\|^{-d/p} \big(\log\|y\|\big)^{-1/p-2\sigma/d}\, dy \\
& \gtrsim \|x\|^{-d/p} \big(\log\|x\|\big)^{-1/p-2\sigma/d}
	\int_{\|x\|/2 < \|y\| < \|x\|} \|x-y\|^{2\sigma-d} \exp\big(-2c\|x-y\| \|x\|\big)\, dy.
\end{align*}
As we shall see in a moment, the last integral is comparable with $\|x\|^{-2\sigma}$. Thus
$$
\mathcal{I}^{\sigma}f(x) \gtrsim \|x\|^{-d/p-2\sigma} \big(\log\|x\|\big)^{-1/p-2\sigma/d}
	= \|x\|^{-d/q} \big(\log\|x\|\big)^{-1/q}, \qquad \|x\| > 2e,
$$
and the claim follows.

It remains to analyze the last integral, which we denote by $\mathcal{J}$. Changing the variable
$y=x - z/\|x\|$ we get
$$
\mathcal{J} = \|x\|^{-2\sigma} \int_{{D}_x} \|z\|^{2\sigma-d} e^{-2c\|z\|}\, dz,
$$
where the set of integration is 
${D}_x = \{z \in \mathbb{R}^d : \|x\|^2/2 < \|\,x\|x\|-z\| < \|x\|^2\}$.
We now observe that ${D}_x$ contains the ball 
$B_x = \{z \in \mathbb{R}^d : \|\,x\|x\|/4-z\| < \|x\|^2/4\}$. Indeed, if $z \in B_x$ then
$$
\frac{\|x\|^2}2 < \Bigg| \bigg\| \frac{x\|x\|}4-z\bigg\| - \bigg\|\frac{3}4 x\|x\|\bigg\| \Bigg|
\le \big\| \,x\|x\|-z\big\| \le \bigg\| \frac{x\|x\|}4 - z\bigg\| + \bigg\|\frac{3}4 x \|x\|\bigg\| 
< \|x\|^2.
$$
Thus we have
$$
\|x\|^{-2\sigma} \int_{{B}_x} \|z\|^{2\sigma-d} e^{-2c\|z\|}\, dz \le \mathcal{J} \le
\|x\|^{-2\sigma} \int_{\mathbb{R}^d} \|z\|^{2\sigma-d} e^{-2c\|z\|}\, dz.
$$
Clearly, the integral over $\mathbb{R}^d$ here is finite. The integral over $B_x$ depends on $x$ only
through $\|x\|$. Since the balls $B_x$ are increasing in the sense of $\subset$ when $x$ is moved
away from the origin along a fixed line passing through the origin, we see that the integral over
$B_x$ is an increasing function of $\|x\|$, which is positive and finite. We conclude that
$\mathcal{J} \simeq \|x\|^{-2\sigma}$, $\|x\| > 1$, as desired.

We pass to weak type and restricted weak type inequalities.
Consider first the three `corners' of the boundary of $R$ 
from the statement of Theorem \ref{thm:LpLq}.
If $(\frac{1}p,\frac{1}q) = (1,1-\frac{2\sigma}d)$, then the weak type
$(1,\frac{d}{d-2\sigma})$ holds by \cite[Theorem 2.3]{NS3}. Notice that this property can be expressed
in terms of Lorentz spaces by saying that $\mathcal{I}^{\sigma}$ is bounded from $L^1(\mathbb{R}^d)$
to $L^{d/(d-2\sigma),\infty}(\mathbb{R}^d)$. Then $(\mathcal{I}^{\sigma})^*$ (the adjoint operator
in the Banach space sense) maps boundedly 
$(L^{d/(d-2\sigma),\infty}(\mathbb{R}^d))^*$ into $(L^1(\mathbb{R}^d))^* = L^{\infty}(\mathbb{R}^d)$. 
Further, the associate space of $L^{d/(d-2\sigma),\infty}(\mathbb{R}^d)$ in the sense
of \cite[Chapter 1, Definition 2.3]{BS} is $L^{d/(2\sigma),1}(\mathbb{R}^d)$
(cf. \cite[Chapter 4, Theorem 4.7]{BS}), and by
\cite[Chapter 1, Theorem 2.9]{BS} it can be regarded as a subspace of the dual of 
$L^{d/(d-2\sigma),\infty}(\mathbb{R}^d)$. Since $(\mathcal{I}^{\sigma})^* = \mathcal{I}^{\sigma}$
by symmetry of the kernel, we infer that $\mathcal{I}^{\sigma}$ is of restricted weak type
$(\frac{d}{2\sigma},\infty)$. On the other hand, weak type $(\frac{d}{2\sigma},\infty)$ coincides,
by definition, with the strong type, so $\mathcal{I}^{\sigma}$ is not of weak type
$(\frac{d}{2\sigma},\infty)$ in view of the strong type results we already know. This clarifies
the situations related to the `corners' $(1,1-\frac{2\sigma}d)$ and $(\frac{2\sigma}d,0)$.

Taking into account $(\frac{1}p,\frac{1}q) = (0,\frac{2\sigma}d)$, we will show that $\mathcal{I}^{\sigma}$
is of weak type $(\infty, \frac{d}{2\sigma})$. To do that, it is enough to verify the estimate
\begin{equation} \label{w}
\big|\big\{x\in \mathbb{R}^d : |\mathcal{I}^\sigma f(x)|>\lambda\big\}\big|
\lesssim \bigg(\frac{\|f\|_{\infty}}{\lambda}\bigg)^{d\slash(2\sigma)}, 
\qquad \lambda>0,\quad f\in L^\infty(\mathbb{R}^d).
\end{equation}
But this is immediate in view of the bound, see Lemma \ref{lem:le3},
$$
\|\mathcal K^\sigma(x,\cdot)\|_1\le C\|x\|^{-2\sigma}, \qquad x\in\mathbb{R}^d,
$$
since then it follows that $|\mathcal{I}^{\sigma}f(x)| \le C \|x\|^{-2\sigma} \|f\|_{\infty}$
and consequently
$$
\big\{x\in \mathbb{R}^d : |\mathcal I^\sigma f(x)|>\lambda\big\}\subset 
\bigg\{x\in \mathbb{R}^d : \|x\|<\bigg(C\frac{\|f\|_\infty}{\lambda}\bigg)^{1/{2\sigma}}\bigg\}.
$$
This inclusion leads directly to \eqref{w}. 

Finally, we disprove the restricted weak type in the two triangles, see Figure \ref{fig1}.
In the lower triangle we use an \emph{au contraire} argument involving an extension of the Marcinkiewicz
interpolation theorem for Lorentz spaces due to Stein and Weiss \cite[Chapter 4, Theorem 5.5]{BS}.
Indeed, if $\mathcal{I}^{\sigma}$ were of restricted weak type $(p,q)$ for some $p$ and $q$ such that
$\frac{1}q < \frac{1}p - \frac{2\sigma}d$, then by interpolation with a strong type pair satisfying
$\frac{1}q = \frac{1}p-\frac{2\sigma}d$, $p>1$, $q < \infty$, $\mathcal{I}^{\sigma}$ would be of strong
type $(\widetilde{p},\widetilde{q})$ for some $\widetilde{p}$ and $\widetilde{q}$ corresponding to a point
in the lower triangle, a contradiction with $(a)$ above.

To treat the upper triangle, we will give an explicit counterexample.
Let for large $r$
$$
f_{r}(y) = \chi_{\{\|y\|<r\}}.
$$
Clearly, we have $\|f_r\|_{p} \simeq r^{d/p}$.
Estimating as in the proof of (b) above, we get
\begin{align*}
\mathcal{I}^{\sigma}f_r(x) & \gtrsim \int_{\|x\|/2 < \|y\| < \|x\|} \|x-y\|^{2\sigma-d}
	\exp\big(-c\|x-y\|(\|x\|+\|y\|)\big) \chi_{\{\|y\|<r\}} \, dy \\
& \ge \chi_{\{\|x\|<r\}}
	\int_{\|x\|/2 < \|y\| < \|x\|} \|x-y\|^{2\sigma-d} \exp\big(-2c\|x-y\| \|x\|\big)\, dy \\
& \gtrsim \chi_{\{1 < \|x\| < r\}} \|x\|^{-2\sigma},
\end{align*}
uniformly in large $r$ and $x \in \mathbb{R}^d$. Consequently,
$$
\big| \big\{ x \in \mathbb{R}^d : \mathcal{I}^{\sigma}f_r(x) > \lambda \big\}\big| \ge
\big|\big\{ 1 < \|x\| < r : \|x\| < (C\lambda)^{-1/(2\sigma)} \big\}\big|
$$
for some $C>0$ independent of $r$ and $\lambda > 0$.
Taking $\lambda = r^{-2\sigma}$ we conclude that the weak type $(p,q)$ estimate for $\mathcal{I}^{\sigma}$
implies $r^d \lesssim r^{dq/p + 2\sigma q}$. 
This bound, however, fails when $\frac{1}q > \frac{1}p + \frac{2\sigma}d$ and $r\to \infty$. 

The proof is finished.
\end{proof}

For completeness, we remark that in the context of Theorem \ref{thm:LpLq} the question of
weak/restricted weak type $(p,q)$ inequalities related to the segment 
$\frac{1}q = \frac{1}p+\frac{2\sigma}d$, $1 \le q < \frac{2\sigma}d$, is more subtle
and remains open. Considering the case $\sigma>d/2$, the operator $\mathcal{I}^\sigma$ 
is bounded from $L^p(\mathbb R^d)$ to $L^q(\mathbb R^d)$ for every $1\le p,q\le\infty$, 
see \cite[Theorem 2.3]{NS3}. 
The behavior of $\mathcal{I}^{\sigma}$ in the limiting case $\sigma = d/2$ is described by the
theorem below. This result enhances \cite[Theorem 2.3]{NS3} when $\sigma = d/2$.
\begin{theor}
Let $d \ge 1$ and $1 \le p,q \le \infty$. Then $\mathcal{I}^{d/2}$ is bounded from 
$L^p(\mathbb{R}^d)$ to $L^q(\mathbb{R}^d)$ except for $(p,q)=(\infty,1)$ and $(p,q)=(1,\infty)$.
Considering the two singular cases, we have:
\begin{itemize}
\item[(i)] $\mathcal{I}^{d/2}$ is of weak type $(\infty,1)$, but not of strong type $(\infty,1)$;
\item[(ii)] $\mathcal{I}^{d/2}$ is not of restricted weak type $(1,\infty)$.
\end{itemize}
\end{theor}

\begin{proof}
The $L^p$-$L^q$ boundedness is contained in \cite[Theorem 2.3]{NS3}.
To show (i), we observe that the weak type $(\infty,1)$ holds true since the proof of \eqref{w} covers
also the case $\sigma = d/2$. The strong type $(\infty,1)$ fails because 
$\mathcal{I}^{d/2}\boldsymbol{1} \notin L^1(\mathbb{R}^d)$, as easily seen by means of Lemma \ref{lem:le3}.

It remains to verify (ii). For $0< \varepsilon < 1/e$, 
let $f_{\varepsilon}(x) = \chi_{\{\|x\|<\varepsilon\}}$. By the lower bound of Theorem \ref{main1}
it follows that
$$
\mathcal{I}^{d/2}f_{\varepsilon}(x) \gtrsim \int_{\|y\|< \varepsilon}
	\log \frac{1}{\|x-y\|(\|x\|+\|y\|)} \, dy, \qquad \|x\|< 1/e,
$$
uniformly in $\varepsilon < 1/e$. Therefore,
$$
\big\|\mathcal{I}^{d/2}f_{\varepsilon}\big\|_{\infty} \gtrsim \int_{\|y\|<\varepsilon} -\log \|y\|\,dy
	= C_d \int_0^{\varepsilon} -r^{d-1}\log r \, dr \gtrsim \varepsilon^d \log\frac{1}{\varepsilon},
		\qquad 0 < \varepsilon < 1/e,
$$
and we conclude that
$$
\frac{\|\mathcal{I}^{d/2}f_{\varepsilon}\|_{\infty}}{\|f_{\varepsilon}\|_1} \gtrsim 
	\log\frac{1}{\varepsilon}, \qquad 0 < \varepsilon < 1/e.
$$
Letting $\varepsilon \to 0^+$, we see that $\mathcal{I}^{d/2}$ is not of restricted weak type $(1,\infty)$.
\end{proof}

%%%%%%%%%%%%%%%%%%%%%%%%%%%%%%%%%%%%%%%%%%%%%%%%%%%%%%%%%%%%%%%%%%%%%%%%%%%%%%%%%%%%%%%%%%%%%%%%%%%%%%%%%%%

\end{document}